\documentclass{amsart}
\usepackage{amssymb}

\newtheorem{theorem}{Theorem}[section]
\newtheorem{lemma}[theorem]{Lemma}

\theoremstyle{definition}
\newtheorem{definition}[theorem]{Definition}

\newtheorem{corollary}[theorem]{Corollary}

\newtheorem{example}[theorem]{Example}

\theoremstyle{remark}
\newtheorem{remark}[theorem]{Remark}

\numberwithin{equation}{section}

\begin{document}

\title{Gluing and Hilbert Functions of Monomial Curves}

\author{Feza Arslan}
\address{Department of Mathematics,
Middle East Technical University, Ankara, 06531 Turkey}
\email{sarslan@metu.edu.tr}
\author{P{\i}nar Mete}
\address{Department of Mathematics, Bal{\i}kesir University, Bal{\i}kesir,
10145 Turkey} \email{pinarm@balikesir.edu.tr}
\author{Mesut \c{S}ah\.{i}n}
\address{Department of Mathematics,
At{\i}l{\i}m University, Ankara, 06836 Turkey}
\email{mesutsahin@gmail.com}

%\thanks{}

\subjclass[2000]{Primary 13H10, 14H20; Secondary 13P10}
\keywords{Hilbert function of local ring, tangent cone, monomial
curve, numerical semigroup, semigroup gluing, nice gluing, Rossi's
conjecture}

%\date{\today}

% at present the "communicated by" line appears only in ERA and PROC
\commby{}

\dedicatory{}

\begin{abstract}
In this article, by using the technique of gluing semigroups, we
give infinitely many families of 1-dimensional local rings with
non-decreasing Hilbert functions. More significantly, these are
local rings whose associated graded rings are not necessarily
Cohen-Macaulay. In this sense, we give an effective technique to
construct large families of 1-dimensional Gorenstein local rings
associated to monomial curves, which support Rossi's conjecture
saying that every Gorenstein local ring has non-decreasing Hilbert
function.
\end{abstract}

\maketitle

\section{Introduction}

In this article, we study the Hilbert functions of local rings
associated to affine monomial curves obtained by using the technique
of gluing numerical semigroups. The concept of gluing was introduced
by J.C. Rosales in \cite{ros} and used by several authors to produce
new examples of set-theoretic and ideal-theoretic complete
intersection affine or projective varieties (for example
\cite{mo-tho, sahin, thoma4}). We give large families of local rings
with non-decreasing Hilbert functions and generalize the results in
\cite{arslan} and \cite{pf} given for nice extensions, which are in
fact special types of gluings. In doing this, we also give the
definition of a nice gluing which is a generalization of a nice
extension. Moreover, by using the technique of nice gluing, we
obtain infinitely many families of 1-dimensional local rings with
non-Cohen-Macaulay associated graded rings and still having
non-decreasing Hilbert functions. We demonstrate that nice gluing is
an effective technique to construct large families of 1-dimensional
Gorenstein local rings associated to monomial curves, which support
the conjecture due to Rossi saying that every Gorenstein local ring
has non-decreasing Hilbert function \cite{pf}.

Our main interest in this article is the following question about gluing:

\begin{quotation}
\textit{Question}. If the Hilbert functions of the local rings
associated to two monomial curves are non-decreasing, is the Hilbert
function of the local ring associated to the monomial curve obtained
by gluing these two monomial curves also non-decreasing?
\end{quotation}

Every monomial curve in affine 2-space is obtained by gluing, and it
is well-known that every local ring associated to a monomial curve
in affine 2-space has a non-decreasing Hilbert function. In affine
3-space, every monomial curve is not obtained by gluing, but every
local ring associated to a monomial curve in affine 3-space has also
a non-decreasing Hilbert function. This follows from the important
result of Elias saying that every one-dimensional Cohen-Macaulay
local ring with embedding dimension three has a non-decreasing
Hilbert function \cite{elias}. Thus, the above question is trivial
for the monomial curves in affine 2-space and 3-space, which are
obtained by gluing, while the question is open even for the monomial
curves in 4-space, which are obtained by gluing. What makes this
question important is that, if the answer is affirmative even in the
case of gluing complete intersection monomial curves, it will follow
that the Hilbert function of every local ring associated to a
complete intersection monomial curve is non-decreasing. This will be
due to a result of Delorme \cite{delorme}, which is restated by
Rosales in terms of gluing and says that every complete intersection
numerical semigroup minimally generated by at least two elements is
a gluing of two complete intersection numerical semigroups
\cite[Theorem 2.3]{ros}. Considering that it is still not known
whether the Hilbert function of local rings with embedding dimension
four associated to complete intersection monomial curves in affine
4-space is non-decreasing, this will be an important step in proving
Rossi's conjecture.

We recall that an \textit{affine monomial curve}
$C(n_{1},\dots,n_{k})$ is a curve with generic zero
$(t^{n_{1}},\dots,t^{n_{k}})$ in the affine $n$-space
$\mathbb{A}^{n}$ over an algebraically closed field $K$, where
$n_{1}<\dots<n_{k}$ are positive integers with ${\rm
gcd}(n_1,n_2,\dots,n_k)=1$ and $\{n_{1},n_2,\dots,n_{k}\}$ is a
minimal set of generators for the numerical semigroup $\langle
n_{1},n_2,\dots,n_{k} \rangle=\{n \mid n=\sum_{i=1}^{k}a_{i}n_{i},
\; a_{i}$'s are non-negative integers\}. The local ring associated
to the monomial curve $C=C(n_{1},\dots,n_{k})$ is
$K[[t^{n_1},\dots,t^{n_k}]]$, and the Hilbert function of this local
ring is the Hilbert function of its associated graded ring
$gr_m(K[[t^{n_1},\dots,t^{n_k}]])$, which is isomorphic to the ring
$K[x_1,\dots,x_k]/I(C)_{*}$, where $I(C)$ is the defining ideal of
$C$ and $I(C)_{*}$ is the ideal generated by the polynomials $f_{*}$
for $f$ in $I(C)$ and $f_{*}$ is the homogeneous summand of $f$ of
least degree. In other words, $I(C)_{*}$ is the defining ideal of
the tangent cone of $C$ at $0$.

\section{Technique of Gluing Semigroups and Monomial curves}

In this section, we first give the definition of gluing for
numerical semigroups.

\begin{definition} \label{gluing} \cite[Lemma 2.2]{ros} Let $S_1$ and
$S_2$ be two numerical semigroups minimally generated by
$m_{1}<\dots<m_{l}$ and $n_{1}<\dots<n_{k}$ respectively. Let
$p=b_1m_1+\dots+b_lm_l \in S_1$ and $q=a_1n_1+\dots+a_kn_k \in S_2$
be two positive integers satisfying $gcd(p,q)=1$ with $p \not\in
\{m_{1},\dots,m_{l}\}$, $q \not\in \{n_{1},\dots,n_{k}\}$ and
$\{qm_{1},\dots,qm_{l}\} \cap \{pn_{1},\dots,pn_{k}\}={\o}$. The
numerical semigroup $S=\langle
qm_{1},\dots,qm_{l},pn_{1},\dots,pn_{k} \rangle$ is called a gluing
of the semigroups $S_1$ and $S_2$.
\end{definition}

Thus, the monomial curve
$C=C(qm_{1},\dots,qm_{l},pn_{1},\dots,pn_{k})$ can be interpreted as
the gluing of the monomial curves $C_1=C(m_1,\dots,m_l)$ and
$C_2=C(n_1,\dots,n_k)$, if $p$ and $q$ satisfy the conditions in
Definition \ref{gluing}. Moreover, from \cite[Theorem 1.4]{ros}, if
the defining ideals $I(C_1) \subset K[x_1,\dots,x_l]$ of $C_1$ and
$I(C_2) \subset K[y_1,\dots,y_k]$ of $C_2$ are generated by the sets
$G_1=\{f_1,\dots,f_s\}$ and $G_2=\{g_1,\dots,g_t\}$ respectively,
then the defining ideal of $I(C) \subset
K[x_1,\dots,x_l,y_1,\dots,y_k]$ is generated by the set
$G=\{f_1,\dots,f_s,g_1,\dots,g_t,x_1^{b_1}\dots
x_l^{b_l}-y_1^{a_1}\dots y_k^{a_k}\}$.

Now, consider the local rings $R_1=K[[t^{m_1},\dots,t^{m_l}]]$,
$R_2=K[[t^{n_1},\dots,t^{n_k}]]$ and
$R=K[[t^{qm_1},\dots,t^{qm_l},t^{pn_1},\dots,t^{pn_k}]]$ associated
respectively to the monomial curves $C_1$, $C_2$ and $C$ obtained by
gluing $C_1$ and $C_2$. Our main interest is whether the Hilbert
function of $R$ is non-decreasing, given that the Hilbert functions
of the local rings $R_1$ and $R_2$ are non-decreasing.

We first answer the following question: If $C_1$ and $C_2$ have
Cohen-Macaulay tangent cones, is the tangent cone of the monomial
curve $C$ obtained by gluing these two monomial curves necessarily
Cohen-Macaulay? The following example shows that the answer is no.

\begin{example} \label{non-cm} Let $C_1$ and $C_2$ be the monomial curves
$C_1=C(5,12)$ and $C_2=C(7,8)$. Obviously, they have Cohen-Macaulay
tangent cones. By a gluing of $C_1$ and $C_2$, we obtain the
monomial curve $C=C(21\times5,21\times12,17\times7,17\times8)$. The
ideal $I(C)$ is generated by the following set $G=\{x_1^{12}-x_2^5,
y_1^8-y_2^7, x_1x_2-y_1^3\}$. The ideal $I(C)_*$ of the tangent cone
of $C$ at the origin is generated by the set $G_*=\{x_1x_2,x_2^5,
y_1^{15},y_2^7,x_2^4y_1^3,x_2^3y_1^6,x_2^2y_1^9,x_2y_1^{12}\}$ which
is a Gr\"{o}bner basis with respect to the negative degree reverse
lexicographical ordering with $x_2>y_2>y_1>x_1$. From \cite[Theorem
2.1]{arslan}, since $x_1$ divides $x_1x_2 \in G_*$, the monomial
curve $C$ obtained by a gluing of $C_1$ and $C_2$ does not have a
Cohen-Macaulay tangent cone. It should also be noted that Hilbert
function of the local ring corresponding to $C$ is non-decreasing,
although the tangent cone of $C$ is not Cohen-Macaulay.
\end{example}

This example leads us to ask the following question:

\begin{quotation}
\textit{Question}. If two monomial curves have Cohen-Macaulay
tangent cones, under which conditions does the monomial curve
obtained by gluing these two monomial curves also have a
Cohen-Macaulay tangent cone?
\end{quotation}

To answer this question partly, we first give the definition of a
nice gluing, which generalizes the definition of a nice extension
given in \cite{pf}.

\begin{definition}
Let $S_1=\langle m_{1},\dots,m_{l} \rangle$ and $S_2=\langle
n_{1},\dots,n_{k} \rangle$ be two numerical semigroups minimally
generated by $m_{1}<\dots<m_{l}$ and $n_{1}<\dots<n_{k}$
respectively. The numerical semigroup $S=\langle
qm_{1},\dots,qm_{l},pn_{1},\dots,pn_{k} \rangle$ obtained by gluing
$S_1$ and $S_2$ is called a nice gluing, if $p=b_1m_1+\dots+b_lm_l
\in S_1$ and $q=a_1n_1 \in S_2$ with $a_1 \leq b_1+\dots+b_l$.
\end{definition}

\begin{remark} \label{extension}
Notice that a nice extension defined in \cite{pf} is exactly a nice
gluing with $S_2=\langle 1 \rangle$.
\end{remark}

\begin{remark} \label{smallest}
It is important to determine the smallest integer among the
generators of the numerical semigroup $S=\langle
qm_{1},\dots,qm_{l},pn_{1},\dots,pn_{k} \rangle$ obtained by gluing,
since this is essential in checking the Cohen-Macaulayness of the
tangent cone of the associated monomial curve. The condition $a_1
\leq b_1+\dots+b_l$ with $m_{1}<\dots<m_{l}$, $n_{1}<\dots<n_{k}$,
${\rm gcd}(p,q)=1$ and $\{qm_{1},\dots,qm_{l}\} \cap
\{pn_{1},\dots,pn_{k}\}={\o}$ implies that
\[qm_1=a_1n_1m_1 \leq (b_1+\dots+b_l)n_1m_1 <
pn_{1}=(b_1m_1+\dots+b_lm_l)n_1\] and $qm_1$ is the smallest integer
among the generators of $S$.
\end{remark}

We can now state the following theorem:

\begin{theorem} \label{maintheorem1}
Let $S_1=\langle m_{1},\dots,m_{l} \rangle$ and $S_2=\langle
n_{1},\dots,n_{k} \rangle$ be two numerical semigroups minimally
generated by $m_{1}<\dots<m_{l}$ and $n_{1}<\dots<n_{k}$, and let
$S=\langle qm_{1},\dots,qm_{l},pn_{1},\dots,pn_{k} \rangle$ be a
nice gluing of $S_1$ and $S_2$. If the associated monomial curves
$C_1=C(m_{1},\dots,m_{l})$ and $C_2=C(n_{1},\dots,n_{k})$ have
Cohen-Macaulay tangent cones at the origin, then
$C=C(qm_{1},\dots,qm_{l},pn_{1},\dots,pn_{k})$ has also
Cohen-Macaulay tangent cone at the origin, and thus, the Hilbert
function of the local ring
$K[[t^{qm_1},\dots,t^{qm_l},t^{pn_1},\dots,t^{pn_k}]]$ is
non-decreasing.
\end{theorem}

To prove this theorem, we first give a refinement of the criterion
for checking the Cohen-Macaulayness of the tangent cone of a
monomial curve given in \cite[Theorem 2.1]{arslan}, which was used
in Example \ref{non-cm}. The advantage of this modification in the
criterion is that instead of first finding the generators of the
tangent cone and then computing another Gr\"obner basis, it only
needs the computation of the standard basis of the generators of the
defining ideal of the monomial curve with respect to a special local
order. Recall that a local order is a monomial ordering with $1$
greater than any other monomial. For the examples and properties of
local orderings, see \cite{singular}. We denote the leading monomial
of a polynomial $f$ by ${\rm LM}(f)$.

\begin{lemma} \label{criterion}
Let $\langle n_{1},\dots,n_{k} \rangle$ be a numerical semigroup
minimally generated by $n_{1}<\dots<n_{k}$, $C=C(n_{1},\dots,n_{k})$
be the associated monomial curve and $G=\{f_1,\dots,f_s\}$ be a
minimal standard basis of the ideal $I(C) \subset K[x_1,\dots,x_k]$
with respect to the negative degree reverse lexicographical ordering
that makes $x_1$ the lowest variable. $C$ has Cohen-Macaulay tangent
cone at the origin if and only if $x_1$ does not divide ${\rm
LM}(f_i)$ for $1 \leq i \leq s$.
\end{lemma}

This lemma combines a result of Bayer-Stillman \cite[Theorem
15.13]{eisenbud} with the well-known fact that a monomial curve
$C=C(n_{1},\dots,n_{k})$, where $n_1$ is smallest among the integers
$n_{1},\dots,n_{k}$,  has Cohen-Macaulay tangent cone if and only if
$x_1$ is not a zero divisor in the ring $K[x_1,\dots,x_k]/I(C)_{*}$
\cite[Theorem 7]{garcia}. \\

\noindent {\em Proof}. Recalling that $f_{*}$ is the homogeneous
summand of the polynomial $f$ of least degree, if $x_1$ divides
${\rm LM}(f_i)$ for some $i$, then either $f_{i_*}=x_1m$ or
$f_{i_*}=x_1m+\sum c_im_i$, where $m_i$'s are monomials having the
same degree with $x_1m$ and $c_i$'s are in $K$. In the latter case,
$x_1$ must divide each $m_i$, because we work with the negative
degree reverse lexicographical ordering that makes $x_1$ the lowest
variable. This implies that in both cases $f_{i_*}=x_1g$ where $g$
is a homogeneous polynomial. Moreover, $g \not\in I(C)_*$. If $g \in
I(C)_*$, then there exists $f \in I(C)$ such that $f_{*}=g$ so ${\rm
LM}(f)={\rm LM}(g)$. Since $\langle {\rm LM}(f_1),\dots,{\rm
LM}(f_s) \rangle= \langle{\rm LM}(I(C))\rangle$, there exists an
$f_j \in G$ such that ${\rm LM}(f_j)$ divides ${\rm LM}(f)={\rm
LM}(g)$ and this contradicts with the minimality of $G$. Thus, $x_1g
\in I(C)_*$, while $g \not\in I(C)_*$, which makes $x_1$ a
zero-divisor in $K[x_1,\dots,x_k]/I(C)_{*}$. Hence, the tangent cone
of the monomial curve $C$ is not Cohen-Macaulay. Conversely, if
$K[x_1,\dots,x_k]/I(C)_{*}$ is not Cohen-Macaulay, then $x_1$ is a
zero-divisor in $K[x_1,\dots,x_k]/I(C)_{*}$. Thus, $x_1m \in
I(C)_*$, where $m$ is a monomial and $m \not\in I(C)_*$. The ideal
generated by the leading monomials of the elements in $I(C)$
obviously contains $x_1m$. Since $G$ is a standard basis, there
exists $f_i \in G$ such that ${\rm LM}(f_i)=x_1m'$, where $m'$
divides $m$ and $m' \not\in I(C)_*$, because $m \not\in I(C)_*$.
This completes the proof.
\begin{flushright}
$\Box$
\end{flushright}

We can now prove Theorem~\ref{maintheorem1}. \\

\noindent {\em Proof of Theorem \ref{maintheorem1}}. By using the
notation in \cite{singular}, we denote the s-polynomial of the
polynomials $f$ and $g$ by $\rm{spoly}(f,g)$ and the Mora's
polynomial weak normal form of $f$ with respect to $G$ by $NF(f
\vert G)$. Let $G_1=\{f_1,\dots,f_s\}$ be a minimal standard basis
of the ideal $I(C_1) \subset K[x_1,\dots,x_l]$ with respect to the
negative degree reverse lexicographical ordering with $x_2 > \cdots
> x_l > x_1$ and $G_2=\{g_1,\dots,g_t\}$ be a minimal standard basis of the ideal
$I(C_2) \subset K[y_1,\dots,y_k]$ with respect to the negative
degree reverse lexicographical ordering with $y_2 > \cdots > y_k >
y_1$. Since $C_1$ and $C_2$ have Cohen-Macaulay tangent cones at the
origin, we conclude from Lemma \ref{criterion} that $x_1$ does not
divide the leading monomial of any element in $G_{1}$ and $y_1$ does
not divide the leading monomial of any element in $G_{2}$ for the
given orderings. The defining ideal of the monomial curve $C$
obtained by gluing is generated by the set
$G=\{f_1,\dots,f_s,g_1,\dots,g_t,x_1^{b_1}\dots
x_l^{b_l}-y_1^{a_1}\}$. Moreover, this set is a minimal standard
basis with respect to the negative degree reverse lexicographical
ordering with $y_2 > \cdots > y_k > y_1 > x_2 > \cdots
> x_l > x_1$, because $NF({\rm spoly}(f_i,g_j) \vert G)=0$, $NF({\rm
spoly}(f_i,x_1^{b_1}\dots x_l^{b_l}-y_1^{a_1}) \vert G)=0$ and
$NF({\rm spoly}(g_j,x_1^{b_1}\dots x_l^{b_l}-y_1^{a_1}) \vert G)=0$
for $1 \leq i \leq s$ and $1 \leq j \leq t$. This is due to the fact
that $NF({\rm spoly}(f,g) \vert G)=0$, if $\rm{lcm}({\rm LM}(f),{\rm
LM}(g))={\rm LM}(f)\cdot {\rm LM}(g)$. From Remark \ref{smallest},
$qm_1$ is the smallest integer among the generators of $G$. Thus,
$C$ has Cohen-Macaulay tangent cone at the origin if and only if
$x_1$, which corresponds to $qm_1$, is not a zero-divisor in
$K[x_1,\dots,x_l,y_1,\dots,y_k]/I(C)_{*}$. Since $x_1$ does not
divide the leading monomial of any element in $G_{1}$ and $G_{2}$,
and ${\rm LM}(x_1^{b_1} \dots x_l^{b_l}-y_1^{a_1})=y_1^{a_1}$, $x_1$
does not divide the leading monomial of any element in $G$, which is
a minimal standard basis with respect to the negative degree reverse
lexicographical ordering with $y_2 > \cdots > y_k > y_1 > x_2 >
\cdots > x_l > x_1$. Thus, from Lemma \ref{criterion}, $C$ has
Cohen-Macaulay tangent cone at the origin.
\begin{flushright}
$\Box$
\end{flushright}

\begin{remark}
From Remark \ref{extension}, every nice extension is a nice gluing.
Thus, if the monomial curve $C=C(m_{1},\dots,m_{l})$ has a
Cohen-Macaulay tangent cone at the origin, then every monomial curve
$C'=C(qm_{1},\dots,qm_{l},b_1m_1+\dots+b_lm_l)$ obtained by a nice
gluing has also Cohen-Macaulay tangent cone at the origin. Thus,
Theorem \ref{maintheorem1} generalizes the results in
\cite[Proposition 4.1]{arslan} and \cite[Theorem 3.6]{pf}.
\end{remark}

\begin{example} Let $C_1$ and $C_2$ be the monomial curves
$C_1=C(m_1,m_2)$ with $m_1 < m_2$ and $C_2=C(n_1,n_2)$ with $n_1 <
n_2$. Obviously, they have Cohen-Macaulay tangent cones. From
Theorem \ref{maintheorem1}, every monomial curve
$C=C(qm_1,qm_2,pn_1,pn_2)$ obtained by a nice gluing with
$q=a_1n_1$, $p=b_1m_1+b_2m_2$, ${\rm gcd}(p,q)=1$ and $a_1 \leq
b_1+b_2$ has Cohen-Macaulay tangent cone at the origin, so the local
ring $R=K[[t^{qm_1},t^{qm_2},t^{pn_1},t^{pn_2}]]$ associated to the
monomial curve $C$ has a non-decreasing Hilbert function. Thus, by
starting with fixed $m_1,m_2,n_1$ and $n_2$, we can construct
infinitely many families of 1-dimensional local rings with
non-decreasing Hilbert functions. For example, consider the monomial
curves $C_1=C(2,3)$ and $C_2=C(4,5)$. By choosing $q=2n_1=8$ and
$p=(2r)m_1+m_2=4r+3$, for any $r \geq 1$, we obtain the monomial
curve $C(16,24,16r+12,20r+15)$, which is a nice gluing of $C_1$ and
$C_2$. Since $C$ is also a complete intersection monomial curve
having a Cohen-Macaulay tangent cone, the associated local rings are
Gorenstein with non-decrasing Hilbert functions. Obviously, they support Rossi's conjecture.
\end{example}

This example shows that gluing is an effective method to obtain new
families of monomial curves with Cohen-Macaulay tangent cones.
Especially in affine 4-space, nice gluing is a very efficent method
to obtain large families of complete intersection monomial curves
with Cohen-Macaulay tangent cones, since every monomial curve in
affine 2-space has a Cohen-Macaulay tangent cone.

\section{Monomial Curves with Non-Cohen-Macaulay Tangent Cones}

In this section we show that nice gluing is not only an efficient
tool to obtain new families of monomial curves with Cohen-Macaulay
tangent cones, but more significantly, it is very useful for
obtaining families of monomial curves with non-Cohen-Macaulay
tangent cones having nondecreasing Hilbert functions. In other
words, it is an effective method to obtain families of local rings
with non-decreasing Hilbert functions. In this sense, it can be used
to obtain families of local rings in proving the conjecture due to
Rossi saying that a one-dimensional Gorenstein local ring has a
non-decreasing Hilbert function.

\begin{theorem} \label{maintheorem2}
Let $S_1=\langle m_{1},\dots,m_{l}\rangle$ and $S_2=\langle
n_{1},\dots,n_{k}\rangle$ be two numerical semigroups minimally
generated by $m_{1}<\dots<m_{l}$ and $n_{1}<\dots<n_{k}$, and let
$S=\langle qm_{1},\dots,qm_{l},pn_{1},\dots,pn_{k} \rangle$ be a
nice gluing of $S_1$ and $S_2$. (Recall that $p=b_1m_1+\dots+b_lm_l
\in S_1$ and $q=a_1n_1 \in S_2$ with $a_1 \leq b_1+\dots+b_l$.) Let
the local ring $K[[t^{m_1},\dots,t^{m_l}]]$ associated to the
monomial curve $C_1=C(m_{1},\dots,m_{l})$ have a non-decreasing
Hilbert function and let $C_2=C(n_{1},\dots,n_{k})$ have
Cohen-Macaulay tangent cone at the origin, then the Hilbert function
of the local ring
$K[[t^{qm_1},\dots,t^{qm_l},t^{pn_1},\dots,t^{pn_k}]]$ associated to
the monomial curve $C=C(qm_{1},\dots,qm_{l},pn_{1},\dots,pn_{k})$
obtained by gluing is also non-decreasing.
\end{theorem}

\begin{proof}
Let $G_1=\{f_1,\dots,f_s\}$ be a minimal standard basis
of the ideal $I(C_1) \subset K[x_1,\dots,x_l]$ with respect to the
negative degree reverse lexicographical ordering with $x_2 > \cdots
> x_l > x_1$ and $G_2=\{g_1,\dots,g_t\}$ be a minimal standard basis of the ideal
$I(C_2) \subset K[y_1,\dots,y_k]$ with respect to the negative
degree reverse lexicographical ordering with $y_2 > \cdots > y_k >
y_1$. Since $C_2$ has Cohen-Macaulay tangent cone, $y_1$ does not divide ${\rm LM}(g_i)$
for $1 \leq i \leq t$ from Lemma \ref{criterion}. From the proof of Theorem \ref{maintheorem1}, $G=\{f_1,\dots,f_s,g_1,\dots,g_t,x_1^{b_1}\dots
x_l^{b_l}-y_1^{a_1}\}$ is a minimal standard
basis with respect to the negative degree reverse lexicographical
ordering with $y_2 > \cdots > y_k > y_1 > x_2 > \cdots
> x_l > x_1$, and again from Lemma \ref{criterion}, we have $\langle
{\rm LM}(I(C)_*)\rangle=\langle {\rm LM}(f_1),\dots,{\rm
LM}(f_s),{\rm LM}(g_1),\dots,{\rm LM}(g_t),y_1^{a_1} \rangle$.
Hence, recalling the well-known result going back to Macaulay
\cite{macaulay}, the Hilbert function of the local ring
$K[[t^{qm_1},\dots,t^{qm_l},t^{pn_1},\dots,t^{pn_k}]]$ is equal to
the Hilbert function of the graded ring
$$R=K[x_1,\dots,x_l,y_1,\dots,y_k]/\langle{\rm LM}(f_1),\dots,{\rm LM}(f_s),{\rm LM}(g_1),\dots,{\rm LM}(g_t),y_1^{a_1}
\rangle.$$ By using \cite[Proposition 2.4]{bayer} and recalling
that $y_1 \nmid {\rm LM}(g_i)$ for $1 \leq i \leq t$, $R$ is
isomorphic to $R_1 \otimes_K R_2 \otimes_K R_3$, where
$R_1=K[x_1,\dots,x_l]/\langle{\rm LM}(f_1),\dots,{\rm
LM}(f_s)\rangle$, $R_2=K[y_2,\dots,y_l]/\langle{\rm
LM}(g_1),\dots,{\rm LM}(g_t)\rangle$ and $R_3=K[y_1]/\langle
y_1^{a_1}\rangle$. Moreover, Hilbert series of $R$ is the product
of Hilbert series of $R_1, R_2$ and $R_3$. Hilbert series of $R_1$
can be given as $h_1(t)/(1-t)$, where the polynomial $h_1(t)$ has
non-negative coefficients, since from the assumption the local
ring associated to the monomial curve $C_1$ has non-decreasing
Hilbert function. The Hilbert series of $R_2$ can be given as
$h_2(t)$, where the polynomial $h_2(t)$ has non-negative
coefficients, because $R_2$ is the Artinian reduction of the
Cohen-Macaulay tangent cone of the monomial curve $C_2$. Observing
that the Hilbert series of $R_3$ is $h_3(t)=1+t+\dots+t^{a_1-1}$,
we obtain that the Hilbert series of $R$ is
$h_1(t)h_2(t)h_3(t)/(1-t)$, where the polynomial
$h_1(t)h_2(t)h_3(t)$ has non-negative coefficients. This proves
that Hilbert function of $R$ is non-decreasing.
\end{proof}

We can now use this theorem to obtain large families of Gorenstein
monomial curves with non-Cohen-Macaulay tangent cones having
nondecreasing Hilbert functions to support Rossi's conjecture.

\begin{example} Let $C_1$ and $C_2$ be the monomial curves
$C_1=C(6,7,15)$ and $C_2=C(1)$. $C_1$ has non-Cohen-Macaulay
tangent cone, having a non-decreasing Hilbert function. Obviously,
they satisfy the conditions of the Theorem \ref{maintheorem2},
which implies that every local ring associated to the monomial
curve $C=C(6q,7q,15q,6q+7)$ obtained by a nice gluing (which is
also a nice extension) with $q \not\equiv 0$ (mod 7) has a
non-decreasing Hilbert function. $C_1$ is a complete intersection
monomial curve, with $I(C_1)=\langle
x_1^5-x_3^2,x_1x_3-x_2^3\rangle$ having a minimal standard basis
with respect to the negative degree reverse lexicographical
ordering with $x_2 > x_3> x_1$ given by
$\{x_1^5-x_3^2,x_1x_3-x_2^3,x_2^3x_3-x_1^6,x_2^6-x_1^7\}$. Hence,
$C$ is a complete intersection monomial curve, with $I(C)=\langle
x_1^5-x_3^2,x_1x_3-x_2^3,y_1^q-x_1^qx_2 \rangle$ having a minimal
standard basis with respect to the negative degree reverse
lexicographical ordering with $y_1>x_2 > x_3> x_1$ given by
$\{x_1^5-x_3^2,x_1x_3-x_2^3,x_2^3x_3-x_1^6,x_2^6-x_1^7,y_1^q-x_1^qx_2\}$,
which shows that $C$ has non-Cohen-Macaulay tangent cone. Thus, we
have obtained Gorenstein local rings
$K[[t^{6q},t^{7q},t^{15q},t^{6q+7}]]$ with $q \not\equiv 0$ (mod
7) having non-Cohen-Macaulay associated graded rings and
non-decreasing Hilbert functions. In this way, starting with a
complete intersection monomial curve $C_1$ in affine 3-space
having non-Cohen-Macaulay tangent cone, we can construct
infinitely many families of 1-dimensional Gorenstein local rings
with non-Cohen-Macaulay associated graded rings and non-decreasing
Hilbert functions. In this way, we can construct infinitely many
families of Gorenstein local rings supporting Rossi's conjecture.
\end{example}

\begin{corollary} Every local ring with embedding dimension 4
associated to a monomial curve obtained by a nice gluing of

a) $C_1=C(m_1,m_2)$ with $m_1 < m_2$ and $C_2=C(n_1,n_2)$ with $n_1
< n_2$

b) $C_1=C(m_1,m_2,m_3)$ with $m_1 < m_2 < m_3$ and
$C_2=C(1)$

c) $C_1=C(1)$ and $C_2=C(n_1,n_2,n_3)$ with $n_1 < n_2 < n_3$, whose tangent cone is Cohen-Macaulay\\

\noindent has a non-decreasing Hilbert function.
\end{corollary}

\begin{proof} In part a), the result follows both from Theorem
\ref{maintheorem1} and Theorem \ref{maintheorem2}. In part b), the
result follows from Theorem \ref{maintheorem2}, since every local
ring associated to the monomial curve $C_1=C(m_1,m_2,m_3)$ has
non-decreasing Hilbert function due to a result of Elias
\cite{elias}. In the same way, in part c), the result is a direct consequence of Theorem \ref{maintheorem2}.
\end{proof}

\section*{Acknowledgements} We would like to thank Apostolos Thoma
and Marcel Morales for mentioning the connection between extension
and gluing. We also would like to thank the referee for very helpful
suggestions.

\end{document}